\documentclass[a4paper,10pt]{amsart}

\usepackage[utf8]{inputenc}
\usepackage[T1]{fontenc}
\usepackage{lmodern}
\usepackage{microtype}

\usepackage{hyperref}

\usepackage[english]{babel}
\usepackage{amsmath}
\usepackage{amssymb}
\usepackage{mathrsfs}
\usepackage{amsthm}
\usepackage[abbrev]{amsrefs}
\usepackage{nameref}
\usepackage{cleveref}
\usepackage{enumitem}
\usepackage{tikz}

\usepackage[left=3.3cm,right=3.3cm,top=4.5cm,bottom=3.5cm]{geometry}

\theoremstyle{plain}
\newtheorem{proposition}{Proposition}[section]
\newtheorem{lemma}[proposition]{Lemma}
\newtheorem{theorem}[proposition]{Theorem}
\newtheorem{corollary}[proposition]{Corollary}
\newtheorem{definition}[proposition]{Definition}

\theoremstyle{definition}
\newtheorem{example}[proposition]{Example}
\newtheorem*{counterexample}{Counterexample}
\theoremstyle{remark}

\crefname{proposition}{proposition}{propositions}
\crefname{lemma}{lemma}{lemmas}
\crefname{theorem}{theorem}{theorems}
\crefname{corollary}{corollary}{corollaries}

\newcommand{\SymmSpace}
	{\mathscr{S}}
\newcommand{\reals}
	{\mathbf{R}}
\newcommand{\integers}
	{\mathbf{N}}
\newcommand{\ambSp}
	{\reals^{d}}
\newcommand{\sphere}
	{{S^{d-1}}}
\newcommand{\Proj}
	{{\mathbf{P}^{d-1}_ {\reals}}}
\newcommand{\inertie}
	{\mathcal{A}}
\newcommand{\admissible}
	{\mathscr{X}}
\newcommand{\SteinerSymm}
	{\mathcal{S}}
\newcommand{\PolarSymm}
	{\mathcal{H}}
\newcommand{\CapSymm}
	{\mathcal{L}}

\newcommand{\controlSet}
	{{\overline{X^\SymmSpace_{\varepsilon}}}}
\newcommand{\support}
	{\textrm{supp}}
\newcommand{\measurable}
	{\mathcal{M}_\sharp}
\newcommand{\limitPoint}
	{{\mathfrak{s}_\star}}
\newcommand{\Haus}
	{\mathscr{H}}
\newcommand{\Sect}
	{{\upharpoonright}}
\DeclareMathOperator{\dif}
	{d\!}
\newcommand{\boundary}[1]
	{{\mathcal{I}^{({#1})}}}

\newcommand{\term}[1]
	{\emph{#1}}

\title[Approximation of symmetrizations by Markov processes]{Approximation of symmetrizations\\by Markov processes}
\date{July 29, 2015}
\subjclass[2010]%
{60D05 (47H20,47H40,60J05)}

\author{Justin Dekeyser}
\address[Justin Dekeyser]{Université~catholique~de~Louvain,
Département~de~Mathématique,\newline
2~Chemin~du~Cyclotron, 1348~Louvain-La-Neuve, \textsc{Belgium}}
\email[Justin Dekeyser]{Justin.Dekeyser@uclouvain.be}

\author{Jean Van Schaftingen}
\address[Jean Van Schaftingen]{Université~catholique~de~Louvain,
Département~de~Mathématique,\newline
2~Chemin~du~Cyclotron, 1348~Louvain-La-Neuve, \textsc{Belgium}}
\email[Jean Van Schaftingen]{Jean.VanSchaftingen@uclouvain.be}

\begin{document}

\begin{abstract}
Under continuity and recurrence assumptions,
we prove that the iteration of successive partial symmetrizations
that form a time-homogeneous Markov process, converges to a symmetrization.
We cover several settings, including
the approximation of the spherical nonincreasing rearrangement by
Steiner symmetrizations, polarizations and cap symmetrizations.
A key tool in our analysis is a quantitative
measure of the asymmetry.
\end{abstract}
\maketitle

\section{Introduction and main results}

\subsection{Approximation by Steiner symmetrizations}

Steiner symmetrizations are measure-preserving transformations of sets
that bring symmetry with respect to one direction $u\in\Proj$ in the Euclidean space~\cite{Baernstein}.
The resulting set $X^u$ is symmetric with respect to the direction $u$.
It was observed in the study of the classical isoperimetric inequality that
any Borel measurable set $X\subseteq\ambSp$ which is left invariant under
all Steiner symmetrizations must be an Euclidean ball centered on the
origin~\cite{Steiner}.

A natural question is whether the spherical nonincreasing rearrangement, which
associates to each Borel measurable set $X$ the unique Euclidean ball $X^\star$ centered on
$0$ and with the same measure as $X$, can be
approximated by Steiner symmetrizations, that is whether there exists a sequence
$(u_n)_{n\in\integers}$ such that the sequence of successive Steiner symmetrizations $(X^{u_1\dots u_n})_{n\in\mathbb{N}}$ converges somehow to the spherical nonincreasing
rearrangement $X^\star$.
Such results have been obtained in order to prove various properties of
symmetrizations~\citelist{\cite{Brascamp}\cite{LiebLoss}*{proof of theorem 3.7}}.
The approximation procedure seems quite robust, and this brings the question whether \emph{random
sequences} of partial symmetrizations approximate symmetrizations.

Independent random symmetrizations of sets and functions were studied
in various settings~\citelist{\cite{vanSchaftApproxSymm}\cite{Volcic}\cite{BurchardFortier}\cite{ManiLevitska}\cite{CoupierDavydov}}, and
rates of convergence were recently discovered~\citelist{\cite{BurchardFortier}*{Corollary~5.4, Proposition~6.2}\cite{CoupierDavydov}*{Theorem~3}}.
A typical result for the convergence of independent Steiner symmetrizations is:
\begin{theorem}\label{mainTheoremAlt}
Let $(S_n)_{n\in\integers}$ be a sequence of independent and $\mu$-identically distributed
sequence of Steiner symmetrizations. We have
	\begin{equation}\label{mainTheoremAlt:Equation}
		\mu( \{ u\in\Proj: m(X\Delta X^u) > 0 \} ) > 0 ,
	\end{equation}
for every Lebesgue measurable set $X\subseteq\ambSp$ of finite measure
with $m(X\Delta X^\star)>0$, if and only if,
for every Lebesgue measurable set $X\subseteq\ambSp$ of finite measure,
the sequence of successive Steiner symmetrizations
$(X^{S_1\dots S_n})_{n\in\integers}$ converges almost-surely in
measure to $X^\star$.
\end{theorem}
In this paper we investigate the approximation by \emph{time-homogeneous
Markov processes}. A stochastic process $(S_n)_{n\in\integers}$ valued in a topological
space $\SymmSpace$ is a time-homogeneous Markov process if there exists a transition function
	\[ P : \SymmSpace \times \mathscr{B}(\SymmSpace) \to [0,1] : (s,A)\mapsto P_s(A) ,\]
satisfying some measurability conditions (see~\cite{Tweedie} or
\cref{subsection.Markov} in this text), and such that
if, for every $n\in\integers$,
and for every Borel measurable set $A\in\mathscr{B}(\SymmSpace)$, we have almost-surely
	\[ \mathbb{P}\{ S_{n+1}\in A \big| S_1,\dots,S_n \} = P_{S_n}(A). \]
The iterated kernels $P^k$ are then defined to satisfy (\cite{Tweedie} or
\cref{subsection.Markov} in this text)
	\[ P^k_{S_n}(A_1\times\dots\times A_k)
		= \mathbb{P}\{ S_{n+1}\in A_1,\dots, S_{n+k}\in A_k \big| S_1,\dots,S_n \},\]
almost-surely
for all $n,k\in\integers$ and all Borel measurable sets $A_1,\dots,A_k\in\mathscr{B}(\SymmSpace)$.
In contrast with processes made up of independent and identically distributed 
variables, successive Steiner symmetrizations that form a Markov process are mutually correlated. In the deterministic case, such a dependence can be an obstruction to
convergence~\citelist{\cite{Klain}\cite{BurchardFortier}\cite{Weth}}.
We obtain the following result for Steiner symmetrizations:
\begin{theorem}\label{mainTheoremSS}
Let $(S_n)_{n\in\integers}$ be a time-homogeneous Markov process
of Steiner symmetrizations with initial distribution $\mu$. If
there exists $\limitPoint\in\Proj$ such that
	\begin{enumerate}[label=(\roman*)]
		\item\label{mainTheoremSS:Recurrence} (\textit{Recurrence}) for every nonempty open set
		$\limitPoint\in\mathcal{O}\subseteq\Proj$, we have
			\[ \mathbb{P}( (S_n)_{n\in\integers}\ \textrm{enters}\ \mathcal{O}\ \textrm{infinitely many often} )
			= 1 ,\]
		\item\label{mainTheoremSS:Continuity} (\textit{Continuity}) for every $n\in\integers$
		and for every open set $\mathcal{O}\subseteq(\Proj)^n$, the map
			\[ s\in\Proj \mapsto P^n_{s}(\mathcal{O}) \]
		is lower semi-continuous at $\limitPoint$,
		\item\label{mainTheoremSS:Discrimination} (\textit{Discrimination}) for every
		Lebesgue measurable set $X\subseteq\ambSp$ of finite measure,
		with $m(X\Delta X^\star)>0$, the process of Steiner symmetrizations starting at $\limitPoint$ reaches in finite time
		the set $\{u\in\Proj:m(X\Delta X^u)>0\}$, that is:
		\[ \sum_{n\in\integers}P_{\limitPoint}^n(\,(\Proj)^{n-1}\times
			\{u\in\Proj: m(X\Delta X^u) > 0 \}\,)
	> 0 ;\]
	\end{enumerate}
then for every Lebesgue measurable set $X\subseteq\ambSp$ of finite measure, the
sequence $(X^{S_1\dots S_n})_{n\in\integers}$ converges in measure to $X^\star$,
almost-surely.
\end{theorem}
The discrimination condition~\ref{mainTheoremSS:Discrimination} in \cref{mainTheoremSS}
is similar to condition~(\ref{mainTheoremAlt:Equation}) in \cref{mainTheoremAlt},
and they are equivalent for independent sequences of random symmetrizations.
A necessary condition
for the conclusion of \cref{mainTheoremSS} to hold is that for each Lebesgue measurable
set $X$ of finite measure and with $m(X\Delta X^\star)>0$, we should have
	\begin{equation}\label{intro.Eqn1}
		\sum_{n\in\integers}\mathbb{P}(\{ m(X\Delta X^{S_n})>0 \} )> 0 .
	\end{equation}
By Fubini's theorem, for each $X$, condition~(\ref{intro.Eqn1}) implies that
condition~\ref{mainTheoremSS:Discrimination} holds for some $s\in\Proj$,
but $s$ may depends on $X$ and $n$. Therefore, condition~(\ref{intro.Eqn1}) is close,
but not equivalent, to condition~\ref{mainTheoremSS:Discrimination} in \cref{mainTheoremSS}.
The initial distribution $\mu$ of the process is only involved in the recurrence
condition~\ref{mainTheoremSS:Recurrence}: the recurrent point $\limitPoint$
is assumed to be deterministic, so that its
existence does not simply follows from the compactness of the projective plane $\Proj$.
The continuity condition in \cref{mainTheoremSS}
is stronger than the usual weak-Feller continuity
at $\limitPoint$, but still weaker than the usual strong Feller-continuity everywhere
(see also~\cite{Tweedie} for definitions, and \cref{proposition.strongFeller} and
discussion below for details). The recurrence condition~\ref{mainTheoremSS:Recurrence} and
continuity condition~\ref{mainTheoremSS:Continuity}
ensure together that the asymptotic behaviour of the
process is closed to an independent process with distribution $P_\limitPoint$.

In the proof of \cref{mainTheoremAlt,mainTheoremSS},
we do not study directly the distance between
sets, in order to prove the convergence.
We rather measure the convergence with an asymmetry, which is a
functional of the form
	\begin{equation}\label{eqn.asymmetryFunction}
		\inertie(X) = \int\limits_X\frac{|x|^2}{1+|x|^2}\dif x .
	\end{equation}
The asymmetry function strictly decreases along Steiner symmetrizations of $X$,
and reaches a minimum at $X^\star$. The idea of using such a function to measure the asymmetry of
sets is a standard technique in the field of symmetrizations (see for example~\citelist{\cite{Shape}\cite{BurchardFortier}\cite{vanSchaftExplicit}\cite{Volcic}}).

\subsection{Other symmetrizations}

The spherical nonincreasing rearrangement has been approximated by other
partial symmetrizations, such as cap symmetrizations~\cite{Sarvas}
and polarizations~\citelist{\cite{BrockSolynin}\cite{Baernstein}\cite{vanSchaftUnivApprox}\cite{vanSchaftApproxSymm}\cite{Weth}}. Other symmetrizations such as
the cap symmetrization~\cite{Smets} and discrete symmetrizations~\cite{Pruss} have also been approximated in the deterministic case in order to
prove isoperimetric theorems.

\Cref{mainTheoremSS} \emph{fails} for the approximation of
the spherical nonincreasing rearrangement by successive cap symmetrizations or
polarizations. Polarizations and cap symmetrizations can be thought of as elements
in $\sphere\times [0,+\infty)$ (see also section~\ref{section.examples} for accurate definitions).
In fact, any cap symmetrization is characterized by a affine half line passing through
the origin $0\in\ambSp$, which yields $\sphere\times [0,+\infty)$ as parametrization
space. Although cap symmetrizations from
$\sphere\times(0,+\infty)$ --~those cap symmetrizations whose half line has edge
different of $0$~--
strictly decrease the asymmetry~\eqref{eqn.asymmetryFunction} of non spherical sets, cap symmetrizations from $\sphere\times\{0\}$
may act as isometries. Such symmetrization could conspire to a non convergent behaviour.
On the other hand, they could also be necessary in the realization of the process (see section~\ref{section.newSymmSpace} for more details).

In order to prevent bad behaviour without just throwing away cap symmetrizations from $\sphere\times\{0\}$, we strengthen
the discrimination condition~\ref{mainTheoremSS:Discrimination} from \cref{mainTheoremSS}. This leads us to the following theorem:
\begin{theorem}\label{mainTheoremCap}
Let $(S_n)_{n\in\integers}$ be a time-homogeneous Markov process
of cap symmetrizations (resp.~polarizations) with initial distribution $\mu$. If
there exists $\limitPoint\in\sphere\times [0,+\infty)$ such that
	\begin{enumerate}[label=(\roman*)]
		\item\label{mainTheoremCap:Recurrence} (\textit{Recurrence}) for every nonempty open set
		$\limitPoint\in\mathcal{O}\subseteq\sphere\times[0,+\infty)$, we have
			\[ \mathbb{P}( (S_n)_{n\in\integers}\ \textrm{enters}\ \mathcal{O}\ \textrm{infinitely many often} )
			= 1 ,\]
		\item\label{mainTheoremCap:Continuity} (\textit{Continuity}) for every $n\in\integers$
		and for every open set $\mathcal{O}\subseteq(\sphere\times[0,+\infty))^n$, the map
			\[ s\in\Proj \mapsto P^n_{\limitPoint}(\mathcal{O}) \]
		is lower semi-continuous at $\limitPoint$,
		\item\label{mainTheoremCap:Discrimination} (\textit{Discrimination})
		for every Lebesgue measurable set $X\subseteq\ambSp$ of finite measure, we have
		\[ \sum_{n\in\integers}P^n_{\limitPoint}(
		 	(\,\sphere\times(0,+\infty)\,)^{n-1}\times\{u\in\sphere\times[0,+\infty):m(X\Delta X^u)> 0\}) > 0 ,
		\]
	\end{enumerate}
then for every Lebesgue measurable set $X\subseteq\ambSp$ of finite measure, the
sequence $(X^{S_1\dots S_n})_{n\in\integers}$ converges in measure to $X^\star$,
almost-surely.
\end{theorem}
While conditions~\ref{mainTheoremCap:Recurrence} and~\ref{mainTheoremCap:Continuity} in \cref{mainTheoremSS} and \cref{mainTheoremCap}
are similar, the discrimination
condition~\ref{mainTheoremCap:Discrimination} takes into account the
bad symmetrizations from $\sphere\times\{0\}$. Condition~\ref{mainTheoremCap:Discrimination} means
that the process starting at $\limitPoint$ reaches in finite time, any set of the form $\{u\in\Proj:m(X\Delta X^u)>0\}$
\emph{without passing through} $\sphere\times\{0\}$.
In the case of independent and identically distributed
cap symmetrizations and polarizations, condition~(\ref{mainTheoremAlt:Equation}) tells us that the boundary set $\sphere\times\{0\}$
is not allowed to support the measure $\mu$; thus the discrimination condition~\ref{mainTheoremCap:Discrimination}
in \cref{mainTheoremCap} reduces to~(\ref{mainTheoremAlt:Equation}).

\subsection{Organization of the paper}

In order to emphasize the main properties of symmetrizations that we use,
we work in \cref{section.settings} with an abstract notion of symmetrizations that covers
Steiner and cap symmetrizations, and polarizations. We draw the reader attention to the fact that the
abstract framework we work with, is mainly aimed to strip the proofs of non pertinent particularities.
Without any assumptions on the process $(S_n)_{n\in\integers}$,
not even the Markov property,
we prove an abstract criterion to test the convergence, \cref{theorem.BF}
(\cref{subsection.abstractSymmProof}). The strategy here is a summability trick used by
Burchard and Fortier~\cite{BurchardFortier}.
This abstract criterion is then particularized in two directions.
We first deduce an abstract version of 
\ref{mainTheoremAlt}. A second particularization is a general result about Markov
processes, from which we derive \cref{mainTheoremSS,mainTheoremCap}.
We finally give several more explicit examples and results.
A particular attention is drawn to a \emph{new example}
for cap symmetrizations, where the boundary $\sphere\times\{0\}$
is needed for the universal convergence and, thus, can not be simply
removed from the parameter space for the cap symmetrizations and polarizations.

\section{Abstract convergence result}\label{section.settings}

\subsection{Abstract symmetrizations}

We fix a metric space $(\admissible,d)$ and a nonexpansive
projector $\star$ in $(\admissible,d)$, that is
a map $[X\in\admissible\mapsto X^\star\in\admissible]$ such that, for all $X,Y\in\admissible$, we have
$d(X^\star,Y^\star)\leq d(X,Y)$ and $X^{\star\star}=X^\star$.
We introduce the following abstract setting for symme\-tri\-za\-tions.
\begin{definition}
A \term{symmetrization space} is a nonempty set $\SymmSpace$ of maps
	$[X\in\admissible \mapsto X^s\in\admissible]$
endowed with a metrizable topology with countable basis, such that
	\begin{enumerate}[label=(\alph*)]
		\item\label{SymmSpace:continuity} (\textit{Continuity}) the map
			$[(X,s)\in\admissible\times\SymmSpace \mapsto X^s]$
		is continuous,
		\item (\textit{Idempotence}) for every $X\in\admissible$, $X^{ss} = X^s$,
		\item\label{SymmSpace:nonexpansiveness} (\textit{Nonexpansiveness})
			for all $X,Y\in\admissible$, $d(X^s,Y^s)\leq d(X,Y)$.
	\end{enumerate}
The elements of $\admissible$ are called \term{objects}, and
elements of $\SymmSpace$ are called \term{symmetrizations}.
\end{definition}
In view of the nonexpansiveness \ref{SymmSpace:nonexpansiveness},
the continuity \ref{SymmSpace:continuity} can be deduced from the
apparently weaker assumption that, for every $X\in\admissible$, the map
$[s\in\SymmSpace \mapsto X^s]$
is continuous.
\begin{definition}
We say that a symmetrization space $\SymmSpace$ is \term{$\star$-compatible} if
	\begin{enumerate}[label=(\alph*)]
		\item for every $s\in\SymmSpace$, for every $X\in\admissible$,
		we have $X^{s\star}=X^\star=X^{\star s}$,
		\item for every $X\in\admissible$, if $X=X^s$ for every $s\in\SymmSpace$,
			then $X=X^\star$.
	\end{enumerate}
\end{definition}
\begin{definition}
Let $\SymmSpace$ be a symmetrization space.
A function $\inertie:\admissible\to\reals$ is an \term{asymmetry} on $\SymmSpace$
if $\inertie$ is continuous and if for every $s\in\SymmSpace$,
for every $X\in\admissible$, we have $\inertie(X^s) \leq \inertie(X)$.
An asymmetry function $\inertie$ is said to be a \term{strict asymmetry}
on $\SymmSpace$
when for every $X\in\admissible$, for every $s\in\SymmSpace$, the equality $\inertie(X^s)=\inertie(X)$ implies $X^s=X$.
\end{definition}
\begin{definition}
Let $\SymmSpace$ be a symmetrization space
and $\inertie$ be an asymmetry function on $\SymmSpace$.
We say that $\inertie$ is \term{$\star$-compatible} if for every $X\in\admissible$ satisfying $\inertie(X)\leq\inertie(X^\star)$, we have $X=X^\star$.
\end{definition}
If a symmetrization space $\SymmSpace$ on $(\admissible,d)$ is $\star$-compatible,
then the function
	\[ \inertie : \admissible \to \reals : \inertie(X) = d(X,X^\star) \]
always defines a $\star$-compatible asymmetry function on $\SymmSpace$.
However, this asymmetry function might not be the best choice in convergence theory.
Another direct consequence of the definitions is that every strict asymmetry function
on a $\star$-compatible symmetrization space $\SymmSpace$, is itself $\star$-compatible.

The next proposition characterizes the convergence in $\admissible$
of iterated symmetrizations in terms of the asymmetry.
\begin{proposition}\label{proposition.ConvergenceCapture}
Let $\SymmSpace$ be a symmetrization space,
$\inertie$ be an asymmetry function on $\SymmSpace$, $(s_n)_{n\in\integers}$ be a sequence in $\SymmSpace$ and $X\in\admissible$.
If $\SymmSpace$ and $\inertie$ are $\star$-compatible,
then
	\[ \liminf_{n\to+\infty} d( X^\star , X^{s_1\dots s_n} ) = 0 \]
if and only if, the set $\{X^{s_1\dots s_n}:n\in\integers\}$ has compact closure
in $\admissible$ and
	\[ \liminf_{n\to+\infty} \inertie( X^{s_1\dots s_n} ) = \inertie(X^\star) . \]
\end{proposition}
\begin{proof}
The ``only if'' part is a consequence of the continuity of $\inertie$, and the fact that the closure of a convergence sequence
is always compact. For the converse, assume that
the set $\{X^{s_1\dots s_n}:n\in\integers\}$ has compact closure, and that
	\[ \liminf_{n\to+\infty} \inertie( X^{s_1\dots s_n} ) = \inertie(X^\star) .\]
Since the asymmetry function decreases along symmetrizations,
the sequence $(\inertie(X^{s_1\dots s_n}))_{n\in\integers}$
converges to $\inertie(X^\star)$.
By compactness assumption, there exists a subsequence $(X^{s_1\dots s_{n_k}})_{k\in\integers}$ that converges to some $Y\in\admissible$.
By continuity of $\inertie$, we have
	\[ \inertie(X^\star) = \lim_{k\to+\infty} \inertie( X^{s_1\dots s_{n_k}} ) = \inertie(Y) .\]
The $\star$-compatibility of $\SymmSpace$ then implies
	\[ d(X^\star,Y^\star) = \lim_{k\to+\infty}d(X^\star,X^{s_1\dots s_{n_k}\star})
		= d(X^\star,X^\star) = 0, \]
so that $X^\star=Y^\star$ and $\inertie(Y^\star)=\inertie(Y)$.
Since $\inertie$ is $\star$-compatible, we deduce that
$Y=X^\star$. Since this is true for each accumulation point of the sequence $(X^{s_1\dots s_n})_{n\in\integers}$, by compactness, this sequence converges to $X^\star$.
\end{proof}

\subsection{Abstract result for random symmetrizations}\label{subsection.abstractSymmProof}

From now on, we fix a probability space $(\Omega,\mathscr{A},\mathbb{P})$ and
$(S_n)_{n\in\integers}$ a sequence of measurable maps from $(\Omega,\mathscr{A})$ to $\SymmSpace$, which is assumed to be a symmetrization space, endowed with its Borel $\sigma$-algebra
$\mathscr{B}(\SymmSpace)$.
For every $n\in\integers$, we write $\mathcal{F}_n$ the sub-$\sigma$-algebra of $\mathscr{A}$
induced by $\{S_1,\dots,S_n\}$,
and $\mathcal{F}$ the smallest sub-$\sigma$-algebra of $\mathscr{A}$
that contains $\bigcup_{n\in\integers}\mathcal{F}_n$.
If $N$ is a stopping time adapted
to $(S_n)_{n\in\integers}$,
its induced filtration is denoted by $\mathcal{F}_N$.
Throughout the text, we write $\mathbb{P}(\cdot)$ (resp.~$\mathbb{E}$)
probabilities (resp.~expectations), and 
$\mathbb{E}\{\cdot|\cdot\}$ conditional expectations.

The next technical lemma allows
to reduce the randomness by taking the infimum; it follows from
the classical properties of conditional expectation.
\begin{lemma}\label{lemma.infGoesOut}
Let $\SymmSpace$ be a symmetrization space
and $\mathcal{B}\subseteq\mathcal{F}$ be a $\sigma$-algebra.
Let be $X\in\admissible$, and set
	\[ X^\SymmSpace = \{ X^{s_1\dots s_n}:n\in\integers, s_1,\dots,s_n\in\SymmSpace \} .\]
If $f: X^\SymmSpace \times \SymmSpace \to \reals$ is continuous
and bounded, and if $\mathfrak{G}:(\Omega,\mathcal{B})\to X^\SymmSpace$ and
$S:(\Omega,\mathcal{F})\to\SymmSpace$ are measurable,
then we have for every $U\in\admissible$, almost-surely on $\mathfrak{G}^{-1}(U)$,
	\[ \mathbb{E}\{ f(\mathfrak{G},S) \big| \mathcal{B} \}
		\geq \inf_{Y\in U}\mathbb{E}\{ f(Y,S) \big| \mathcal{B} \} .\]
\end{lemma}
\begin{proof}
Without loss of generality, we can assume that $f$ is positive.
The topological space $X^\SymmSpace\times\SymmSpace$ is second-countable.
Therefore it is not difficult to prove that,
since $f$ is bounded and continuous on $X^\SymmSpace\times\SymmSpace$,
for all probability measure $\mu$ on $\mathscr{B}(X^\SymmSpace\times\SymmSpace)$,
there exists a decreasing sequence $(f_n)_{n\in\mathbf{N}}$ of simple functions,
converging $\mu$-almost-everywhere to $f$, and whose level sets are finite unions
of disjoint Borel rectangles.

We apply this approximation scheme with the conjoint distribution $\mu$ of the random
vector $(\mathfrak{G},S)$. Let $(f_n)_{n\in\mathbf{N}}$ be the corresponding approximation
sequence. According to standard properties of the conditional
expectation~\cite{Billingsley}*{Theorem~34.3}, we have for every $n\in\integers$,
almost-surely on $\mathfrak{G}^{-1}(U)$,
	\[ \mathbb{E}\{ f_n(\mathfrak{G},S) \big| \mathcal{B} \}
		\geq \inf_{Y\in U}\mathbb{E}\{ f_n(Y,S) \big| \mathcal{B} \} .\]
It now follows from the monotone convergence theorem for the
conditional expectation that, almost-surely on $\mathfrak{G}^{-1}(U)$,
	\begin{multline*}
		\mathbb{E}\{ f(\mathfrak{G},S) \big| \mathcal{B} \}
			= \inf_{n\in\integers}\mathbb{E}\{ f_n(\mathfrak{G},S) \big| \mathcal{B} \} 
			\geq \inf_{n\in\integers}\inf_{Y\in U}\mathbb{E}\{ f_n(Y,S) \big| \mathcal{B} \} \\
			\geq \inf_{n\in\integers}\inf_{Y\in U}\mathbb{E}\{ f(Y,S) \big| \mathcal{B} \}
			= \inf_{Y\in U}\mathbb{E}\{ f(Y,S) \big| \mathcal{B} \}.
			\qedhere
	\end{multline*}
\end{proof}
\noindent We are now ready to prove the main result about general stochastic processes of symmetrizations, which need not be Markov
processes.
\begin{proposition}[Convergence by divergence]\label{theorem.BF}
Let $\SymmSpace$ be a symmetrization space
and $\inertie$ be an asymmetry function on $\SymmSpace$, such that $\SymmSpace$ and $\inertie$ are
both $\star$-compatible.
Let be $X\in\admissible$. If the set
	\[ X^\SymmSpace = \{ X^{s_1\dots s_n}:n\in\integers, s_1,\dots,s_n\in\SymmSpace
		\} \]
has compact closure in $(\admissible,d)$, and
if there exists an increasing and almost-surely finite sequence of
stopping times $(N_n)_{n\in\integers}$ adapted to $(S_n)_{n\in\integers}$,
such that we have almost-surely for all $\epsilon>0$
	\begin{equation}
		\sum_{n\in\integers} \inf
		\Big\{ \mathbb{E}\{ 
			\inertie( Y )- \inertie( Y^{s_{N_n+1}\dots s_{N_{(n+1)}}} )
		\big|\mathcal{F}_{N_n} \} \;:
		 Y\in\overline{X^\SymmSpace},\;
		 \inertie(Y)\geq \inertie(X^\star)+\varepsilon \Big\} = +\infty, \label{eqn.BF}
	\end{equation}
then the sequence $(X^{S_1\dots S_n})_{n\in\integers}$ converges 
almost-surely to $X^\star$.
\end{proposition}
The proof is based on a summability trick found by Burchard and Fortier~\cite{BurchardFortier}.
\begin{proof}
If $X=X^\star$, there is nothing to prove. Otherwise,
by $\star$-compatibility of
$\inertie$, there exists $\varepsilon>0$
such that $\inertie(X)\geq \inertie(X^\star)+\varepsilon$.
For $m\in\integers$, we write almost-surely
	\begin{align*}
		\inertie(X)-\inertie(X^\star) &= \inertie(X)-\inertie(X^{S_1\dots S_{N_1}}) + \inertie(X^{S_1\dots S_{N_{(m+1)}}})-\inertie(X^\star) \\
			&\quad + \sum^m_{n=1} \inertie( X^{S_1\dots S_{N_n}} ) - \inertie( X^{S_1\dots S_{N_{(n+1)}}} ) \\
			&\geq \sum^m_{n=1} \inertie( X^{S_1\dots S_{N_n}} ) - \inertie( X^{S_1\dots S_{N_{(n+1)}}} ) \\ 
			&\geq \sum^m_{n=1} \chi_{\Theta_n}\cdot ( \inertie( X^{S_1\dots S_{N_n}} ) - \inertie( X^{S_1\dots S_{N_{(n+1)}}} ) ) .
	\end{align*}
where, for every $n\in\integers$, we have defined the set
	\[ \Theta_n^\varepsilon = \{ \inertie( X^{S_1\dots S_{N_n}} ) \geq \inertie(X^\star)+\varepsilon \} \in \mathcal{F}_{N_n} .\]
(Here the symbol $\chi$ denotes indicator functions of sets.)
Taking the expectation on both sides, we compute
	\begin{align*}
		\inertie(X)-\inertie(X^\star) &\geq \sum^m_{n=1} \mathbb{E}( \chi_{\Theta_n^\varepsilon}\cdot ( \inertie( X^{S_1\dots S_{N_n}} ) - \inertie( X^{S_1\dots S_{N_{(n+1)}}} ) ) ) \\
			&= \sum^m_{n=1} \mathbb{E}( \chi_{\Theta_n^\varepsilon}
				\mathbb{E}\{ \inertie( X^{S_1\dots S_{N_n}} ) - \inertie( X^{S_1\dots S_{N_{(n+1)}}} ) \big| \mathcal{F}_{N_n} \} ) .
	\end{align*}
According to \cref{lemma.infGoesOut}, and writing
$\controlSet = \{ Y\in\overline{X^\SymmSpace}:
		\inertie(Y)\geq\inertie(X^\star)+\varepsilon \}$,
we get
	\begin{align*}
		\inertie(X)-\inertie(X^\star)
			&\geq \sum^m_{n=1} \mathbb{E}( \chi_{\Theta_n^\varepsilon}
			\inf_{Y\in\controlSet}
				\mathbb{E}\{ \inertie(Y) - \inertie( Y^{S_{N_n+1}\dots S_{N_{(n+1)}}}) \big| \mathcal{F}_{N_n} \} ) \\
			&= \mathbb{E}\left( \sum^m_{n=1} \chi_{\Theta_n^\varepsilon}
			\inf_{Y\in\controlSet}
				\mathbb{E}\{ \inertie(Y) - \inertie( Y^{S_{N_n+1}\dots S_{N_{(n+1)}}}) \big| \mathcal{F}_{N_n} \} \right) .
	\end{align*}
Letting $m\in\integers$ tend to $+\infty$, the monotone convergence theorem ensures
	\[ \inertie(X)-\inertie(X^\star) \geq \mathbb{E}\left( \sum_{n\in\integers} \chi_{\Theta_n^\varepsilon} \inf_{Y\in\controlSet}
				\mathbb{E}\{ \inertie(Y) - \inertie( Y^{S_{N_n+1}\dots S_{N_{(n+1)}}}) \big| \mathcal{F}_{N_n} \} \right) . \]
Therefore, we have almost-surely
	\[ \sum_{n\in\integers} \chi_{\Theta_n^\varepsilon} \inf_{Y\in\controlSet}
				\mathbb{E}\{ \inertie(Y) - \inertie( Y^{S_{N_n+1}\dots S_{N_{(n+1)}}} ) \big| \mathcal{F}_{N_n} \} < +\infty .\]
According to identity (\ref{eqn.BF}) and the monotonicity of $\inertie$ along symmetrizations,
the sequence $(\chi_{\Theta_n^\varepsilon})_{n\in\integers}$
reaches $0$ almost-surely after finitely many steps, so that we have almost-surely
	\[ \lim_{n\to\infty}
		\inertie(X^{S_1\dots S_n})\leq\inertie(X^\star)+\varepsilon .\]
Now, considering only rational $\varepsilon>0$, we deduce that
the sequence $(\inertie(X^{S_1\dots S_n}))_{n\in\integers}$ converges almost-surely to
$\inertie(X^\star)$. By \cref{proposition.ConvergenceCapture},
the sequence $(X^{S_1\dots S_n})_{n\in\integers}$ converges almost-surely to $X^\star$.
\end{proof}

\subsection{Convergence of independent processes}\label{subsection.indepProof}

In the context of processes made up from independent and
identically distributed symmetrizations, we obtain an abstract result
from which \cref{mainTheoremAlt} directly follows.

\begin{theorem}\label{theoreme.iid}
Let $\SymmSpace$ be a $\star$-compatible symmetrization space.
Let $(S_n)_{n\in\integers}$ be a sequence of independent and $\mu$-identically distributed
variables. If, for every $X\in\admissible$, the set
	\[ X^\SymmSpace = \{ X^{s_1\dots s_n}:n\in\integers, s_1,\dots,s_n\in\SymmSpace \} \]
has compact closure in $(\admissible,d)$, then
the following conditions are equivalent:
	\begin{enumerate}[label=(\roman*)]
		\item for every $X\in\admissible$, the sequence $(X^{S_1\dots S_n})_{n\in\integers}$
		converges almost-surely to $X^\star$,
		\item for every $X\in\admissible$ with $X\neq X^\star$,
		we have $\mu( \{s\in\SymmSpace: X\neq X^s \} ) > 0$.
	\end{enumerate}
\end{theorem}
\begin{proof}
Let us first assume that, for every $X\in\admissible$ with $X\neq X^\star$, we have
	\[ \mu( \{ s\in\SymmSpace :	X\neq X^s \} ) > 0 .\]
We apply \cref{theorem.BF} with the $\star$-compatible asymmetry
	\[ \inertie : \admissible \to \reals : \inertie(X) = d(X,X^\star) .\]
We consider the increasing sequence of deterministic
stopping times $(N_n)_{n\in\integers}=(n)_{n\in\integers}$.
Let be $X\in\admissible$ with $d(X,X^\star)\geq \varepsilon>0$
(otherwise there is nothing to prove)
and set
	\[ \overline{X^\SymmSpace_\varepsilon}
		= \{ Y \in \overline{X^\SymmSpace} : d(Y,X^\star)\geq \varepsilon \} .\]
\Cref{theorem.BF} shows that it is sufficient to prove
	\[ \inf_{Y\in\overline{X^\SymmSpace_\varepsilon}}
		\int\limits_\SymmSpace
			d(Y,Y^\star) - d(Y^s,Y^\star) \mu(\dif s)
	> 0  .\]
By continuity of $\inertie$ and compactness of $\overline{X^\SymmSpace}$, it suffices to prove that for all
$Y\in\overline{X^\SymmSpace_\varepsilon}$, we have
	\[ \int\limits_\SymmSpace
			d(Y,Y^\star) - d(Y^s,Y^\star) \mu(\dif s)
	> 0  .\]
Let us fix $Y\in\overline{X^\SymmSpace_\varepsilon}$.
By assumption, there exists some $\delta>0$ such that
	\[ \mu( \{ s\in\SymmSpace : d(Y,Y^\star)\geq \delta \} ) > 0 .\]
Therefore, we have
	\[ \int\limits_\SymmSpace
			d(Y,Y^\star) - d(Y^s,Y^\star) \mu(\dif s)
	\geq \delta\mu( \{ s\in\SymmSpace : d(Y,Y^\star)\geq \delta \} ) > 0 .\]
This concludes the first part of the alternative.

For the converse, assume that for every $X\in\admissible$,
the sequence $(X^{S_1\dots S_n})_{n\in\integers}$ converges almost-surely
to $X^\star$. Assume by contradiction
that there exists $X\in\admissible$, $X\neq X^\star$, such that
	\[ \mu( \{ s\in\SymmSpace : X\neq X^s \} ) = 0 \]
or, equivalently,
	\[ \mathbb{P}(\{
		\omega\in\Omega : X=X^{S_1(\omega)} \} ) = 1 .\]
Then the sequence $(X^{S_1\dots S_n})_{n\in\integers}$ is almost-surely
constant and equals $X$, hence $X=X^\star$ by assumption
on $(S_n)_{n\in\integers}$.
\end{proof}

\subsection{Markov processes}\label{subsection.Markov}

In order to deal with Markov processes, we recall some classical
terminology about transition functions.
In a metrizable topological space $\SymmSpace$ with countable basis,
a transition function on $\SymmSpace$ is a function
	\[ P : \SymmSpace \times \mathscr{B}(\SymmSpace) \to [0,1] : (s,A)\mapsto P_s(A) \]
such that
	\begin{enumerate}[label=(\alph*)]
		\item for every $s\in\SymmSpace$, the function
			\[ P_s : \mathscr{B}(\SymmSpace) \to [0,1] : A\mapsto P_s(A) \]
		is a probability measure,
		\item for every Borel measurable and bounded function
		$f:\SymmSpace\to\reals$, the function
			\[ Pf : \SymmSpace \to [0,1] :
					s\mapsto P_sf = \int\limits_\SymmSpace f(y)\,P_s(\dif y) \]
		is (bounded and) Borel measurable.
	\end{enumerate}
For every $n\in\integers$
and every rectangle $A_1\times\dots\times A_n\in\mathscr{B}(\SymmSpace)^n$, the iterated kernel $P^n$ is
	\begin{equation}\label{iteratedKernels1}
		P^n(A_1\times\dots\times A_n) : \SymmSpace \to \reals : s \mapsto
		\int\limits_{\SymmSpace}\cdots\int\limits_{\SymmSpace}
		\Big( \prod_{i=1}^n\chi_{A_i}(s_i) \Big)
			P_{s_{n-1}}(\dif s_n)\cdots\,P_{s_1}(\dif s_2)\,P_s(\dif s_1)
	\end{equation}
We recall~\cite{Tweedie}*{Chapter~3} that the stochastic process $(S_n)_{n\in\integers}$
is a time-homogeneous Markov process on $\SymmSpace$ if
there exists a transition function $P$ on $\SymmSpace$ such that
for all $n,k\in\integers$ with $k>1$,
for every Borel set $A_1\times\dots\times A_k\subseteq\SymmSpace^k$,
we have (almost-surely)
	\begin{equation}\label{iteratedKernels2}
	\mathbb{E}\{ \chi_{A_1\times\dots\times A_k}(S_{n+1},\dots,S_{n+k}) \big| \mathcal{F}_n \}
		= P^n_{S_n}(A_1\times\dots\times A_k).
	\end{equation}
Since we deal with discrete time processes, this equality extends to stopping times. Furthermore, identities (\ref{iteratedKernels1})
and (\ref{iteratedKernels2}) directly extend to bounded and continuous
functions $f:\SymmSpace^n\to\reals$, following the same approximation scheme of \cref{lemma.infGoesOut}.
\begin{proposition}\label{proposition.markov}
Let $\SymmSpace$ be a $\star$-compatible symmetrization space
such that for every $X\in\admissible$, the set
	\[ X^\SymmSpace = \{ X^{s_1\dots s_n}:n\in\integers, s_1,\dots, s_n\in\SymmSpace \} \]
has compact closure in $(\admissible,d)$.
Let $\inertie$ be a $\star$-compatible asymmetry function on $\SymmSpace$.
Assume that there exists $\boundary{\SymmSpace}\in\mathscr{B}(\SymmSpace)$ such that
$\inertie$ is strict on $\SymmSpace\setminus\boundary{\SymmSpace}$.
If there exists $\limitPoint\in\SymmSpace$ such that
	\begin{enumerate}[label=(\roman*)]
		\item\label{recurrence.proposition.markov} (\textit{Recurrence}) for every
		nonempty open set
		$\limitPoint\in\mathcal{O}\subseteq\SymmSpace$, we have
			\[ \mathbb{P}( (S_n)_{n\in\integers}\ \textrm{enters}\ \mathcal{O}\ \textrm{infinitely many often} )
			= 1 ,\]
		\item\label{continuity.proposition.markov} (\textit{Continuity})
		for every $n\in\integers$, for every bounded and continuous function
		$f:\SymmSpace^n\to\reals$, the function $P^nf$ is continuous
		at $\limitPoint$,
		\item\label{discrimination.proposition.markov} (\textit{Discrimination})
		for every $X\in\admissible$ with $X\neq X^\star$, we have
		\[  \sum_{n\in\integers}P^n_\limitPoint(
		 	(\SymmSpace\setminus\boundary{\SymmSpace})^{n-1}\times
		 	\{s\in\SymmSpace: X\neq X^s \}) > 0 , \]
	\end{enumerate}
then for every $X\in\admissible$, the sequence
$(X^{S_1\dots S_n})_{n\in\integers}$ converges almost-surely to $X^\star$.
\end{proposition}
\begin{proof}
We apply \cref{theorem.BF}.
Let us consider a sequence $(\mathcal{O}_n)_{n\in\integers}$ of
nonempty open sets in $\SymmSpace$
decreasing to the limit point $\limitPoint$, that is
	\[ \bigcap_{n\in\integers}\mathcal{O}_n=\{\limitPoint\} .\]
We define the stopping time
$N_1 = \min\{ k\in\integers : S_k\in\mathcal{O}_1 \}$,
and for every $n\in\integers$,
	\[ N_{n+1} = \min\{ k\in\integers : k\geq N_n+n , S_k\in\mathcal{O}_{n+1} \}.\]
By recurrence assumption \ref{recurrence.proposition.markov},
the sequence $(N_n)_{n\in\integers}$ is a sequence of stopping times
which is almost-surely finite and satisfies almost-surely $S_{N_n}\in\mathcal{O}_n$
and $N_{n+1}-N_n\geq n$. Fix $\omega\in\Omega$ such that the previous relations hold, and write $(s_n)_{n\in\integers}=(S_{N_n}(\omega))_{n\in\integers}$ and
$\ell_n=N_{n+1}(\omega)-N_n(\omega)$ for every $n\in\integers$.
The sequence $(s_n)_{n\in\integers}$ converges to $\limitPoint$.
Fix $X\in\admissible$ with $\inertie(X)\geq\inertie(X^\star)+\varepsilon$
(otherwise there is nothing to prove).
By \cref{theorem.BF}, we only need to show
	\begin{multline}
		\sum_{n\in\integers} \inf
		\Big\{ \int\limits_{\SymmSpace^{\ell_n}}
			\inertie( Y )- \inertie( Y^{u_1\dots u_{\ell_n}} )
		\,P^{\ell_n}_{s_n}(\dif u_1,\dots,\dif u_{\ell_n}) \,:\\[-1em]
		Y\in\overline{X^\SymmSpace},\,\inertie(Y)\geq \inertie(X^\star)+\varepsilon \Big\} = +\infty, \label{eqn1.proposition.markov}
	\end{multline}
By compactness of $\controlSet$, for every $n\in\integers$, there exists $Y_n\in\controlSet$ such that
	\begin{align*}
		\inf_{Y\in\controlSet}
			\int\limits_{\SymmSpace^{\ell_n}}
			\inertie(Y) &- \inertie(Y^{u_1\dots u_{\ell_n}})\,P^{\ell_n}_{s_n}(\dif u_1,\dots, \dif u_{\ell_n}) \\
		&=\int\limits_{\SymmSpace^{\ell_n}}
			\inertie(Y_n) - \inertie(Y_n^{u_1\dots u_{\ell_n}})\,P^{\ell_n}_{s_n}(\dif u_1,\dots,\dif u_{\ell_n})
	\end{align*}
Since $\overline{X^\SymmSpace_\varepsilon}$ is compact, there exists a subsequence
$(Y_{n_k})_{k\in\integers}$ that converges to some $Y\in\overline{X^\SymmSpace_\varepsilon}$.
Without loss of generality, we can assume
	\begin{align*}
		\int\limits_{\SymmSpace^{\ell_{n_k}}}
			\inertie(Y_{n_k}) &- \inertie(Y_{n_k}^{u_1\dots u_{\ell_{n_k}}})\,P^{\ell_{n_k}}_{s_{n_k}}(\dif u_1,\dots,\dif u_{\ell_{n_k}}) \\
	&\geq \int\limits_{\SymmSpace^{\ell_{n_k}}}
			\inertie(Y) - \inertie(Y^{u_1\dots u_{\ell_{n_k}}})\,P^{\ell_{n_k}}_{s_{n_k}}(\dif u_1,\dots,\dif u_{\ell_{n_k}}) - \frac{1}{2^k} .
	\end{align*}
It is now sufficient to check that
	\[ \sum_{k\in\integers}
		\int\limits_{\SymmSpace^{\ell_{n_k}}}
			\inertie(Y) - \inertie(Y^{u_1\dots u_{\ell_{n_k}}})\,P^{\ell_{n_k}}_{s_{n_k}}(\dif u_1,\dots,\dif u_{\ell_{n_k}})
		= +\infty .\]
By the continuity assumption \ref{continuity.proposition.markov}, for all
$k\in\integers$, there exists a smaller integer $j_k\in\integers$
such that for all $j\geq j_k$,
	\begin{align*}
		\bigg| \int\limits_{\SymmSpace^{\ell_{n_k}}} \inertie(Y)
			&- \inertie(Y^{u_1\dots u_{\ell_{n_k}}})\,P^{\ell_{n_k}}_{s_j}(\dif u_1,\dots,\dif u_{\ell_{n_k}}) \\
		&- \int\limits_{\SymmSpace^{\ell_{n_k}}} \inertie(Y)
			- \inertie(Y^{u_1\dots u_{\ell_{n_k}}})\,P^{\ell_{n_k}}_{\limitPoint}(\dif u_1,\dots,\dif u_{\ell_{n_k}}) \bigg|
		\leq \frac{1}{2^k} .
	\end{align*}
Define $m_1=\min\{n_k:n_k\geq j_1\}$, and by recurrence
	\[ m_{k+1}= \min\{ n_k : n_k\geq\max\{m_k+1,j_{k+1},n_{(k+1)}\} \} . \]
By construction, $(m_k)_{k\in\mathbf{N}}$ is a
subsequence of $(n_k)_{n\in\mathbf{N}}$ such that $m_k\geq j_k$
for all $k\in\mathbf{N}$.
Since the asymmetry decreases along symmetrizations, we have for every $k\in\integers$
	\begin{align*}
		\int\limits_{\SymmSpace^{\ell_{m_k}}} \inertie(Y)
			& - \inertie(Y^{u_1\dots u_{\ell_{m_k}}}) \,P^{\ell_{m_k}}_{s_{m_k}}(\dif u_1,\dots,\dif u_{\ell_{m_k}}) \\
			&\geq \int\limits_{\SymmSpace^{\ell_{n_k}}}
			\inertie(Y) - \inertie(Y^{u_1\dots u_{\ell_{n_k}}}) \,P^{\ell_{n_k}}_{s_{m_k}}(\dif u_1,\dots,\dif u_{\ell_{n_k}}) \\
			&\geq \int\limits_{\SymmSpace^{\ell_{n_k}}}
			\inertie(Y) - \inertie(Y^{u_1\dots u_{\ell_{n_k}}}) \,P^{\ell_{n_k}}_{\limitPoint}(\dif u_1,\dots,\dif u_{\ell_{n_k}})
			 - \frac{1}{2^k} \\
			&\geq \int\limits_{\SymmSpace^k}
			\inertie(Y) - \inertie(Y^{u_1\dots u_k}) \,P^k_{\limitPoint}(\dif u_1,\dots,\dif u_k)
			 - \frac{1}{2^k} .
	\end{align*}
In the previous line, we have used the fact
that $\ell_{n_k}\geq n_k\geq k$ by construction.
If we prove the strict inequality
	\begin{equation}\label{eqn2.proposition.markov}
		\sup_{n\in\integers}\int\limits_{\SymmSpace^n}
		\inertie(Y) - \inertie(Y^{u_1\dots u_n})\, P^n_{\limitPoint}(\dif u_1,\dots,\dif u_n) > 0,
	\end{equation}
then by comparison of series, condition (\ref{eqn1.proposition.markov}) would then hold.
Let us thus prove (\ref{eqn2.proposition.markov}), where $Y\in\admissible$
satisfies $Y\neq Y^\star$.
By the discrimination assumption \ref{discrimination.proposition.markov},
there exists $n\in\integers$ such that
	\[ P^n_\limitPoint((\SymmSpace\setminus\boundary{\SymmSpace})^{n-1}\times \{s\in\SymmSpace\setminus\boundary{\SymmSpace}:Y\neq Y^s \} )
	> 0 .\]
Since $\inertie $ is strict on $\SymmSpace\setminus\boundary{\SymmSpace}$,
there exists $\delta>0$ such that $P^n_\limitPoint(H)>0$,
with
	\[ H = (\SymmSpace\setminus\boundary{\SymmSpace})^{n-1}\times
		 	\{s\in\SymmSpace:\inertie(Y)\geq\inertie(Y^s)+\delta \} . \]
Let us assume by contradiction that we have
	\[ \int\limits_{\SymmSpace^n}
		\inertie(Y)-\inertie(Y^{s_1\dots s_n})\,P^n_\limitPoint(\dif s_1,\dots,\dif s_n) = 0 .\]
There exists a set $E\in\mathscr{B}(\SymmSpace^n)$
of $P^n_\limitPoint$-measure $1$ such that,
for every $(s_1,\dots,s_n)\in E$, we have
$\inertie(Y) = \inertie(Y^{s_1\dots s_n})$.
For $(s_1,\dots,s_n)\in E$, with $s_1\notin\boundary{\SymmSpace}$,
we have $\inertie(Y)=\inertie(Y^{s_1})$ and,
since $\inertie$ is a strict asymmetry function on $\SymmSpace\setminus\boundary{\SymmSpace}$,
we have $Y = Y^{s_1}$, so that we also have $\inertie(Y)=\inertie(Y^{s_2})$.
By recurrence, we have $Y = Y^{s_i},i=1,\dots,n-1$,
for every $(s_1,\dots,s_n)\in E$ with $s_i\notin\boundary{\SymmSpace}$ for each $i\in\{1,\dots,n-1\}$.
We now have
	\begin{align*}
		0 &= \int\limits_{\SymmSpace^n}
			\inertie(Y)-\inertie(Y^{s_1\dots s_n})\,P^n_\limitPoint(\dif s_1,\dots,\dif s_n) \\
		&\geq \int\limits_{H\cap E}
			\inertie(Y)-\inertie(Y^{s_n})\,P^n_\limitPoint(\dif s_1,\dots,\dif s_n)
			\geq \delta P^n_\limitPoint(H) > 0,
	\end{align*}
which is the desired contradiction. This proves
	\[ \sup_{n\in\integers} \int\limits_{\SymmSpace^n}
		\inertie(Y)-\inertie(Y^{s_1\dots s_n})\,P^n_\limitPoint(\dif s_1,\dots,\dif s_n)  > 0 ,\]
and (\ref{eqn2.proposition.markov}) holds.
\end{proof}
In practice, the proof of the existence of a recurrent point $\limitPoint$ requires some
additional work. We recall a notion of stability that is useful for locally compact symmetrization spaces.
\begin{definition}[\cite{Tweedie}*{Section~9.2}]
Let $\SymmSpace$ be a locally compact metrizable topological space with countable basis,
and $(S_n)_{n\in\integers}$ a Markov process on $\SymmSpace$. The process $(S_n)_{n\in\integers}$ is said to be \term{nonevanescent} if
	\[ \mathbb{P}(
		\exists K\subseteq\SymmSpace : K\ \textrm{compact}, (S_n)_{n\in\integers}
			\ \textrm{enters}\ K\ \textrm{infinitely many often}
	 ) = 1 .\]
\end{definition}
\noindent The topological assumptions on $\SymmSpace$ ensure that the above definition
makes sense.
\begin{proposition}[\cite{Tweedie}*{Theorem~9.1.3}]\label{proposition.markovLimitTheorem}
Let $\SymmSpace$ be a metrizable topological space with countable basis,
and $(S_n)_{n\in\integers}$ be a Markov process on $\SymmSpace$.
For every $A\in\mathscr{B}(\SymmSpace)$, the sequence 
$(\inf_{k\in\integers}P^k_{S_n}(\,(\SymmSpace\setminus A)^k))_{n\in\integers}$
converges almost-surely to the indicator function of the set
	\[ \bigcup_{n\in\integers}\bigcap_{k\geq n}\{S_{k+1}\notin A\} .\]
\end{proposition}
\begin{proof}
We fix a Borel measurable set $A\in\mathscr{B}(\SymmSpace)$, and we consider the function
	\[ H : \SymmSpace \mapsto [0,1] : H(s) = 1 - \inf_{k\in\integers}
		P^k_{s}(\,(\SymmSpace\setminus A)^k\, ) .\]
We write, for every $n\in\integers$,
	\[ \Omega_n = \bigcup_{i\geq n}\{S_{i+1}\in A\} \in \mathcal{F} ,\]
and
	\[ \Omega_\infty = \bigcap_{m\in\integers}\bigcup_{i\geq m}\{S_{i+1}\in A\} .\]
For every $n\in\integers$, the Markov property ensures that almost-surely
	$ H(S_n) = \mathbb{E}\{ \chi_{\Omega_n} \big| \mathcal{F}_n \}$.
Moreover, we also have for each $m\leq n$
	\[ \mathbb{E}\{ \chi_{\Omega_\infty}\big| \mathcal{F}_n \}
		\leq H(S_n) \leq \mathbb{E}\{ \chi_{\Omega_m}\big| \mathcal{F}_n \} .\]
For a fixed $m\in\integers$, the martingale convergence theorem ensures that
the left side converges almost-surely to $\mathbb{E}\{\chi_{\Omega_\infty}\big|\mathcal{F}\}=\chi_{\Omega_\infty}$, and
that the right side converges almost-surely to
$\mathbb{E}\{\chi_{\Omega_m}\big|\mathcal{F}\}=\chi_{\Omega_m}$ as $n\to+\infty$.
Therefore, we have
	\[ \chi_{\Omega_\infty} \leq \liminf_{n\to+\infty} H(S_n) \leq \limsup_{n\to+\infty} H(S_n)
		\leq \chi_{\Omega_m} ,\]
almost-surely for every $m\in\integers$. Letting $m\in\integers$ tends to $+\infty$, we
have almost-surely
	\[ \lim_{n\to+\infty}\inf_{k\in\integers}
		P^k_{S_n}(\,(\SymmSpace\setminus A)^k\, )
	= 1 - \lim_{n\to+\infty}H(S_n) = 1 - \chi_{\Omega_\infty} .\qedhere \]
\end{proof}
\begin{corollary}\label{corollary.markov.quick}
Let $\SymmSpace$ be a $\star$-compatible and locally compact symmetrization space
such that for every $X\in\admissible$, the set
	\[ X^\SymmSpace = \{ X^{s_1\dots s_n}:n\in\integers, s_1,\dots, s_n\in\SymmSpace \} \]
has compact closure in $(\admissible,d)$.
Let $\inertie$ be a $\star$-compatible asymmetry function on $\SymmSpace$.
Assume that there exists $\boundary{\SymmSpace}\subset\SymmSpace$ such that
$\inertie$ is strict on $\SymmSpace\setminus\boundary{\SymmSpace}$.
Let $P:\SymmSpace\times\mathscr{B}(\SymmSpace)\to[0,1]$ be a
transition function, and
$(S_n)_{n\in\integers}$ be a time-homogeneous Markov process with transition function
$P$. If
	\begin{enumerate}[label=(\roman*)]
		\item (\textit{Stability}) the process $(S_n)_{n\in\integers}$ is
		nonevanescent and $\boundary{\SymmSpace}$ is closed,
		\item (\textit{Continuity})
		for all $n\in\integers$, for every bounded and continuous function
		$f:\SymmSpace^n\to\reals$, the function $P^nf$ is continuous,
		\item (\textit{Discrimination})
		for every $s\in\SymmSpace$, for every nonempty open set $\mathcal{O}\subseteq\SymmSpace$,
		\[  \sum_{n\in\integers}P^n_s(
		 	(\SymmSpace\setminus\boundary{\SymmSpace})^{n-1}\times \mathcal{O} ) > 0 , \]
	\end{enumerate}
then for every $X\in\admissible$, the sequence
$(X^{S_1\dots S_n})_{n\in\integers}$ converges almost-surely to $X^\star$.
\end{corollary}
\begin{proof}
Let us fix $\limitPoint\in\SymmSpace$.
According to \cref{proposition.markov}, it is sufficient to
check that for every nonempty open set $\mathcal{O}\subseteq\SymmSpace$
containing $\limitPoint$, we have
	\[ \mathbb{P}( \{
		(S_n)_{n\in\integers}\ \textrm{enters}\ \mathcal{O}
			\ \textrm{infinitely many often}
	 \} ) = 1 .\]
We thus fix a nonempty open set $\mathcal{O}\subseteq\SymmSpace$ that contains $\limitPoint$, and we define the function
	\[ H : \SymmSpace \mapsto [0,1] : H(s) = 1 - \inf_{k\in\integers}
		P^k_{s}(\,(\SymmSpace\setminus\mathcal{O})^k\, ) .\]
By \cref{proposition.markovLimitTheorem}, we have almost-surely
$\lim_{n\to+\infty}H(S_n) = \chi_{\Omega_\infty}$,
where
	\[ \Omega_\infty
			= \bigcap_{m\in\integers}\bigcup_{i\geq m}\{S_{i+1}\in\mathcal{O}\} 
			= \{ (S_n)_{n\in\integers}\ \textrm{enters}\ \mathcal{O}
					\ \textrm{infinitely many often} \} .
	\]
Assume by contradiction that there exists $s\in\SymmSpace$ with $H(s)=0$, that is:
for every $k\in\integers$, we have $P^k_s(\,(\SymmSpace\setminus\mathcal{O})^k\,) = 1$.
Since there exists $n\in\integers$ such that
$P^n_s(\,(\SymmSpace\setminus\boundary{\SymmSpace})^{n-1}\times\mathcal{O}\,) > 0$,
it follows that
	\[
		0 < P^n_s(\,((\SymmSpace\setminus\boundary{\SymmSpace})^{n-1}\times\mathcal{O})
			\cap (\SymmSpace\setminus\mathcal{O})^n \,)
		= P^n_s(\emptyset) = 0 ,
	\]
which is a contradiction. This proves that the function $H$ is strictly positive.
Using Urysohn's lemma, the function $H$ is also lower semi-continuous as supremum of semi-continuous functions. Therefore, for
all compact set $K\subseteq\SymmSpace$, the function $H$ attains a strictly positive
minimal value on $K$, which has to equal $1$. Hence, we have proven the essential inclusion
	\[ \{ (S_n)_{n\in\integers}\ \textrm{enters}\ K\ \textrm{infinitely many often} \}
	 \subseteq \Omega_\infty .\]
Since there exists a countable basis of nonempty open sets with compact closure for
$\SymmSpace$, the nonevanescence of $(S_n)_{n\in\integers}$ ensures
	\[ \mathbb{P}( \{
		(S_n)_{n\in\integers}\ \textrm{enters}\ \mathcal{O}
			\ \textrm{infinitely many often}
	 \} ) = 1 . \qedhere \]
\end{proof}
In view of examples, we prove that a strong Feller transition function $P$ always satisfy our
continuity assumption for $\{P^n:n\in\integers\}$.
\begin{definition}
Let $\SymmSpace$ be a topological space.
A transition function $P$ on $\SymmSpace$ is \term{strong Feller continuous}
at $\limitPoint$
if, for every bounded and Borel measurable function $f:\SymmSpace\to\reals$, the map
$[s\in\SymmSpace\mapsto P_sf]$ is continuous at $\limitPoint$.
\end{definition}
\begin{proposition}[\cite{Tweedie}*{Proposition~6.1.1}]\label{proposition.strongFeller}
Let $\SymmSpace$ be a metrizable topological space with countable
basis, $\limitPoint\in\SymmSpace$ and $P$ a transition function on $\SymmSpace$.
If $P$ is strong Feller continuous at $\limitPoint$, then for every $n\in\integers$,
for every bounded and continuous function
$f:\SymmSpace^n\to\integers$, the function $P^nf$ is continuous at $\limitPoint$.
\end{proposition}
\begin{proof}
If $f:\SymmSpace^n\to\reals$ is bounded and continuous, the function
	\[ g : \SymmSpace\to\reals : g(x) = P^{n}_xf(\cdot,x) \]
is bounded and Borel measurable. The measurability follows from the monotone class
theorem. Since there holds
	\[ P^{n+1}_sf = \int\limits_\SymmSpace g(x)\,P_s(\dif x) ,\]
it suffices to prove that for every bounded and Borel measurable function $g:\SymmSpace\to\reals$,
$Pg$ is continuous at $\limitPoint$. Let $M>0$ be a bound for $g$,
and define $h_1=g+M$. Then $h_1$ is Borel measurable, positive and bounded.
There exists an increasing sequence $(\phi_n)_{n\in\integers}$ of simple functions
that converges to $h_1$, and such that the level sets of $\phi_n$ are
disjoint Borel measurable sets. By linearity, the functions $P\phi_n$ are continuous at $\limitPoint$,
for every $n\in\integers$. Since the sequence $(P\phi_n)_{n\in\integers}$ increases
to $Ph_1$ by the monotone convergence theorem, the function
$Ph_1$ is lower semi-continuous, and so is $Pg$. The same conclusions hold
for the function $h_2=-g+M$, so that $Ph_2$ is lower semi-continuous, and thus
$Pg$ is also upper semi-continuous.
\end{proof}
The converse of \cref{proposition.strongFeller}
is false in general. Consider for example a continuous map $\phi:\SymmSpace\to\SymmSpace$
which is not trivial, and the transition function
	\[ P : \SymmSpace\times\mathscr{B}(\SymmSpace) \to [0,1] :
		P_s(A) = \delta_{\phi(s)}(A) .\]
The iterated kernels are given by the formula
	\[ P^n_sf = f(\,\phi(s),\dots,\phi^n(s)\,) .\]
The functions $P^nf$ are continuous whenever $f$ is itself continuous and bounded,
but the strong Feller continuity may fail in general.

\section{Examples}

In this section, we give examples for various symmetrizations as Steiner and cap
symmetrizations and polarizations. We first recall standard definitions,
and then we give examples of application of our abstract method.

\subsection{Various symmetrizations}\label{section.examples}

We write $\Haus^k$ the $k$-dimensional Hausdorff measure in $\ambSp$.
We work within the metric space
$(\measurable(\ambSp),d_1)$ of equivalence classes of
Borel measurable subsets of $\ambSp$ with finite $\Haus^d$ measure, endowed with the $\Haus^d$-metric. The space $\ambSp$ is equipped with its usual metric.

\vspace{0.5em}\noindent\textbf{Spherical nonincreasing rearrangement.}\ 
The simplest symmetrization transform sets into balls.
\begin{definition}
Let $A\in\mathscr{B}(\ambSp)$ be a Borel measurable set.
The spherical nonincreasing rearrangement of $A$ is the open ball
$A^\star$ centered on the origin $0$, that satisfies $\Haus^d(A^\star) = \Haus^d(A)$.
\end{definition}
The induced map on the quotient space $\measurable(\ambSp)$ is denoted by $\star$. The
fact that $\star$ is an involution is direct. For the nonexpansiveness, one can simply observe
that since $\star$ is measure-preserving and monotone, we directly have for all $A,B\in\mathscr{B}(\ambSp)$
	\begin{align*}
		\Haus^d(A^\star\Delta B^\star)
			&= \Haus^d(A^\star\setminus B^\star) + \Haus^d(B^\star\setminus A^\star) \\
			&\leq \Haus^d(A^\star\setminus (B\cap A)^\star) + \Haus^d(B^\star\setminus (A\cap B)^\star) \\
			&= \Haus^d(A^\star) + \Haus^d(B^\star) - 2\Haus^d((B\cap A)^\star) \\
			&= \Haus^d(A) + \Haus^d(B) - 2\Haus^d(B\cap A) = \Haus^d(A\Delta B) .
	\end{align*}

\vspace{0.5em}\noindent\textbf{Steiner symmetrizations.}\ 
Let be $u\in\Proj$ and $\langle u\rangle$ be its linear span. We write $u^\perp$ the orthogonal complement subspace in $\ambSp$. For every $x\in u^\perp$, we write the section
	\[ A\Sect_x = A\cap (x+\langle u\rangle).\]
\begin{definition}
Given $u\in\Proj$ and a Borel measurable set
$A\in\mathscr{B}(\ambSp)$,
the Steiner symmetrization of $A$ with respect
to $u$ is the unique $A^u\in\mathscr{B}(\ambSp)$ such that
	\begin{itemize}
		\item for every $x\in u^\perp$, $A^u\Sect_x$
		is an open ball of $(x+\langle u\rangle)$,
		\item for every $x\in u^\perp$, $A^u\Sect_x$ is centered on $0$,
		\item for every $x\in u^\perp$,
			$\Haus^1(A^u\Sect_x) = \Haus^1(A\Sect_x)$.
	\end{itemize}
\end{definition}
The induced maps on the quotient $\measurable(\ambSp)$ are the usual Steiner symmetrizations.
The set $\SteinerSymm$ of Steiner symmetrizations being in one-to-one
correspondence with $\Proj$, we equip it with the induced topology of $\Proj$, so that
$\SteinerSymm$ is a metrizable and compact topological space with countable basis.

\vspace{0.5em}\noindent\textbf{Polarizations.}\ 
Fix $\hat{e}\in\sphere$ and $r\geq 0$. The corresponding
affine half-subspace is defined by
	\[ H^{\hat{e},r} = \{ x\in\ambSp : x\cdot \hat{e} \leq r \} .\]
There is a unique nontrivial reflection $\sigma^{\hat{e},r}$ of $\ambSp$ that leaves the boundary of
$H^{\hat{e},r}$ invariant.
\begin{definition}
Given $\hat{e}\in\sphere$, $r\geq 0$ and a Borel measurable set $A\in\mathscr{B}(\ambSp)$, the polarization
of $A$ is defined as the unique $A^{\hat{e},r}\in\mathscr{B}(\ambSp)$ satisfying the following axioms:
	\begin{itemize}
		\item if $x\in H^{\hat{e},r}$, then $x\in A^{\hat{e},r}$ if, and only if,
		$x\in H^{\hat{e},r}\cup\sigma^{\hat{e},r}(H^{\hat{e},r})$,
		\item if $x\notin H^{\hat{e},r}$, then $x\in A^{\hat{e},r}$ if, and only if,
		$x\in H^{\hat{e},r}\cap\sigma^{\hat{e},r}(H^{\hat{e},r})$.
	\end{itemize}
\end{definition}
The induced maps on the quotient $\measurable(\ambSp)$ are the usual polarizations.
The set $\PolarSymm$ of polarizations is in one-to-one
correspondence with $\sphere\times[0,+\infty)$. We equip it with the induced topology of $\sphere\times[0,+\infty)$, so that
$\PolarSymm$ is a metrizable and locally compact topological space with countable basis.
Through this identification,
we define the closed subset
	\[ \boundary{\PolarSymm} = \sphere\times\{0\} .\]
It corresponds to affine half-subspaces that contains $0$ in their usual boundary.

\vspace{0.5em}\noindent\textbf{Cap symmetrizations.}\ 
Fix $\hat{e}\in\sphere$ and $r\geq 0$.
For $t\geq 0$, define the spherical section of a set $A\subseteq\ambSp$ as
	\[ A\Sect_t = A \cap \partial B(r\hat{e} , t) .\]
\begin{definition}
Given $r\geq 0$ and a Borel measurable set
$A\in\mathscr{B}(\ambSp)$, the cap symmetrization of $A$ with respect
to $(\hat{e},r)$ is the unique $A^{\hat{e},r}\in\mathscr{B}(\ambSp)$ that satisfies the following axiom:
	\begin{itemize}
		\item for every $t\geq 0$, $A^{\hat{e},r}\Sect_t$
		is an open ball of $\partial B(r\hat{e}, t)$,
		\item for every $x\in u^\perp$, $A^{\hat{e},r}\Sect_t$ is centered on $(r-t)\hat{e}$
		\item for every $t\geq 0$,
			$\Haus^{d-1}(A^{\hat{e},r}\Sect_t) = \Haus^{d-1}(A\Sect_t)$.
	\end{itemize}
\end{definition}
The induced maps on the quotient $\measurable(\ambSp)$
form the usual set of cap symmetrizations.
The set $\CapSymm$ of cap symmetrizations being in one-to-one
correspondence with $\sphere\times[0,+\infty)$, we equip it with the induced topology of $\sphere\times[0,+\infty)$, so that
$\CapSymm$ is a metrizable and locally compact topological space with countable basis.
We define the closed subset
	\[ \boundary{\CapSymm} = \sphere\times\{0\} .\]
It corresponds to cap symmetrizations with respect to half-lines whose initial points
are the origin $0$.

\vspace{0.5em}\noindent\textbf{Common properties of the examples of symmetrizations.}\ 
Steiner symmetrizations, cap symmetrizations and polarizations enjoy important
common properties, which we recall in the next proposition. The following result
being classical in the field of symmetrizations, we omit the proof.
\begin{proposition}\label{propertiesOfSymm}
The sets $\SteinerSymm,\PolarSymm$ and $\CapSymm$ acting on $(\measurable(\ambSp),d_1)$
are $\star$-compatible symmetrization spaces and for every $X\in\measurable(\ambSp)$,
the sets
	\begin{align*}
		&X^\SteinerSymm = \{ X^{s_1\dots s_n} : n\in\integers,
		s_1,\dots,s_n\in\SteinerSymm \} ,\\
	&\quad X^\PolarSymm = \{ X^{s_1\dots s_n} : n\in\integers,
		s_1,\dots,s_n\in\PolarSymm \} ,\\
	&\quad\quad X^\CapSymm = \{ X^{s_1\dots s_n} : n\in\integers,
		s_1,\dots,s_n\in\CapSymm \} 
	\end{align*}
have compact closure in $(\measurable(\ambSp),d_1)$. The function
	\[ \inertie : \measurable(\ambSp) \to \reals^+ :
		\inertie(X) = \int\limits_X\frac{|x|^2}{1+|x|^2}\dif x \]
is a $\star$-compatible asymmetry function on $\SteinerSymm$, on $\PolarSymm$ and on $\CapSymm$.
Moreover, $\inertie$ is a strict asymmetry function on $\SteinerSymm$, on $\PolarSymm\setminus\boundary{\PolarSymm}$
and on $\CapSymm\setminus\boundary{\CapSymm}$.
\end{proposition}
The previous proposition is straightforward for polarizations~\citelist{\cite{BrockSolynin}\cite{BurchardFortier}*{Polarization identity}}.
The compactness property follows from the Kolmogorov-Riesz compactness theorem.
Once one has the result for polarizations, one can extend it to Steiner and cap symmetrizations by an approximation
argument \cite{vanSchaftExplicit}.

\subsection{Markov Steiner symmetrizations}

In view of the discussion above, we can prove directly \cref{mainTheoremSS}.

\begin{proof}[Proof of \cref{mainTheoremSS}]
The function
	\[ \inertie : \measurable(\ambSp)\to\reals^+:
		\inertie(X) = \int\limits_X\frac{|x|^2}{1+|x|^2}\,dx \]
is a strict asymmetry function on $\SteinerSymm$, according to \cref{propertiesOfSymm}.
The fact that the continuity condition
\ref{mainTheoremSS:Continuity} in \cref{mainTheoremSS} is equivalent to the continuity
condition \ref{continuity.proposition.markov} of \cref{proposition.markov} is a
consequence of Urysohn's lemma. We can apply \cref{proposition.markov} to get the desired
result.
\end{proof}
\begin{example}[Random walk in $\Proj$]
We fix $r>0$ and we write $\sigma$ for the
Haar measure on $\Proj$. The map $[e\in\Proj\mapsto\sigma(B(e,r))]$
is constant. We fix $e\in\Proj$ and we define the transition function
	\[ P: \Proj \times\mathscr{B}(\Proj) \to [0,1] : P_s(A)=\int\limits_{B(s,r)}\chi_A(x)\,\frac{\sigma(\dif x)}{\sigma(B(e,r))} .\]
The transition function $P$ is strong Feller continuous everywhere on $\Proj$.
According to \cref{proposition.strongFeller}, the family $\{P^n:n\in\integers\}$
enjoys the usual continuity assumption at every point.
We also have, for every nonempty open set $\mathcal{O}\subseteq\Proj$ and for every $s\in\Proj$,
	\[ \sum_{n\in\integers}P^n_s(\,(\Proj)^{n-1}\times\mathcal{O}\,) > 0 .\]
This last inequality can be proven by noting that,
for every $n\in\integers$, for every $s\in\Proj$,
the probability measure
	\[ H^n_s : \mathscr{B}(\Proj) \to [0,1] : H^n_s(A) = P^n_s(\,(\Proj)^{n-1}\times A\,) ,\]
has $\overline{B(s,nr)}$ as support.
Since $\Proj$ is compact, any Markov process on $\Proj$ is nonevanescent.
\Cref{corollary.markov.quick} ensures that any time-homogeneous
Markov process $(S_n)_{n\in\integers}$ with transition function $P$
satisfies that for every $X\in\measurable(\ambSp)$, the sequence $(X^{S_1\dots S_n})_{n\in\integers}$ converges almost-surely to $X^\star$.
\end{example}
\begin{proposition}[Deterministic Steiner symmetrizations]
Let $\phi:\Proj\to\Proj$ be a continuous map and $s\in\Proj$ such that
	\[ \overline{\{\phi^n(s):n\in\integers\}}= \Proj .\]
For every $X\in\measurable(\ambSp)$, the sequence $(X^{\phi(s)\dots\phi^n(s)})_{n\in\integers}$ converges
in measure to $X^\star$.
\end{proposition}
\begin{proof}
Let us define for all $n\in\integers$, $S_n = \phi^n(S)$.
The sequence $(S_n)_{n\in\integers}$ is a time-homogeneous Markov process
on $\Proj$. The iterated kernels are given for every $n\in\integers$ by the formula
	\[ P^n_sf=f(\,\phi(s),\dots,\phi^n(s)\,) .\]
The recurrence condition \ref{recurrence.proposition.markov}
and the discrimination condition \ref{discrimination.proposition.markov}
of \cref{proposition.markov}
both follow from the assumption that the orbit of $s$ under $\phi$ is dense in $\Proj$.
Since $\phi$ is continuous,
the continuity condition \ref{continuity.proposition.markov} is satisfied
whenever $f$ is continuous and bounded.
We can thus apply \cref{proposition.markov} with limit
point $\limitPoint=s$. The proof is done.
\end{proof}
The previous proposition can be thought of as a generalization of~\cite{Shape}*{Theorem~5.1},
where the authors studied Kronecker sequence of deterministic Steiner symmetrizations
of the form $(\,(\cos(n\alpha),\sin(n\alpha))\,)_{n\in\integers}$,
with $\alpha/\pi$ irrational. Their analysis is based on the convergence in shape.
Our result can be applied for other deterministic sequence,
and has straightforward generalizations for cap symmetrizations and polarizations.
\begin{counterexample}
The continuity condition \ref{mainTheoremSS:Continuity}
in \cref{mainTheoremSS} is necessary.
Let us consider a sequence $(\alpha_n)_{n\in\integers}$ in $\Proj$ such that
$\alpha_i\neq\alpha_j$ for $i\neq j$, and
such that the set $\{\alpha_n:n\in\integers\}$ is dense in $\Proj$.
We define the transition function on $\Proj$ through
	\[ P_{s} = \left\{ \begin{array}{ll}
		\delta_{\alpha_{n+1}}
			&\quad\textrm{if}\ s=\alpha_n\ \textrm{for some}\ n\in\integers,\\
		\delta_{\alpha_1}
			&\quad\textrm{otherwise}
	\end{array} \right.. \]
If $(S_n)_{n\in\integers}$ is a time-homogeneous Markov process associated with the transition function $P$, then almost-surely
$S_n=\alpha_n$ for every $n\in\integers$. Hence, it is a straightforward computation to check that the hypothesis of \cref{mainTheoremSS}
are satisfied with $\limitPoint=\alpha_1$, except that we miss the continuity property.
Actually, in general, there exists $X\in\measurable(\ambSp)$ such that $(X^{S_1\dots S_n})_{n\in\integers}$ fails to convergence to $X^\star$ almost-surely~\citelist{\cite{Klain}\cite{Weth}}.

In this example, note that one could endow $\Proj$ with the
discrete topology. The continuity assumption is then trivial,
but the recurrence condition forces the process to
have a finite cycle. This is the situation of
iterated Steiner symmetrizations using a finite number of directions~\cite{Klainfini}.
\end{counterexample}

\subsection{Markov cap symmetrizations and polarizations on $(0,+\infty)$}

\begin{proof}[Proof of \cref{mainTheoremCap}]
Since we identify $\CapSymm$ and $\PolarSymm$ with $\sphere\times[0,+\infty)$,
let us recall that $\boundary{\CapSymm}=\sphere\times\{0\}=\boundary{\PolarSymm}$.
We consider the usual asymmetry function
	\[ \inertie : \measurable(\ambSp)\to\reals^+:
		\inertie(X) = \int\limits_X\frac{|x|^2}{1+|x|^2}\,dx .\]
It is a $\star$-compatible asymmetry function on $\CapSymm$ and $\PolarSymm$ which is
strict on the subsets $\CapSymm\setminus\boundary{\CapSymm}$ and $\PolarSymm\setminus\boundary{\PolarSymm}$ (\cref{propertiesOfSymm}).
In these settings, the result follows from \cref{proposition.markov}.
\end{proof}
For polarizations, the following
elementary result shows that the nonevanescence assumption in
\cref{corollary.markov.quick} is actually required for the sequence to
approximate the spherical nonincreasing rearrangement.
This property fails in general for cap symmetrizations.
\begin{proposition}
Let $(s_n)_{n\in\integers}$ be a sequence of polarizations
written for every $n\in\integers$ as $s_n=(\hat{e}_n,r_n)$ with $e_n\in\sphere$
and $r_n\in [0,+\infty)$. If
	$\liminf_{n\to+\infty}r_n>0$,
then there exists $X\in\measurable(\ambSp)$ of finite measure such that
the sequence $(X^{s_1\dots s_n})_{n\in\integers}$ does not converge.
\end{proposition}
\begin{proof}
The condition $\liminf_{n\to+\infty}r_n>0$ shows that there exists $N\in\integers$ such
that, for every $n\geq N$, we have $r_n\geq \delta$, for some $\delta>0$.
We define \[r=\min\{\delta,\min\{r_i:i\in\{1,\dots,N-1\},r_i>0\}\},\]
and $X\in\measurable(\ambSp)$ the
(equivalent class of the) annulus
	\[ X = B(0,r)\setminus B(0,r/2) .\]
Then clearly $X\neq X^\star$. For $n\in\integers$, two cases may occur.
If $r_n\geq r$, then $X$ is contained in the half-space $(\hat{e}_n,r_n)$, and
thus $X=X^{(\hat{e}_n,r_n)}$. If $r_n<r$, then we must have by construction $r_n=0$.
Since $X$ is radially symmetric, we have $X=X^{(\hat{e}_n,0)}$, so that $X$ remains
fixed by all polarizations of the set $\{s_n:n\in\integers\}$.
\end{proof}
\begin{example}[Brownian motion on $(0,+\infty)$]
Let $(Z_n)_{n\in\integers}$ be a sequence of
independent and identically distributed variables on $(0,+\infty)$
with density function $\rho$, and define for every $n\in\integers$
	\[ \left\{
		\begin{array}{ll}
			W_1 &= Z_1, \\
			W_{n+1} &= Z_{n+1}\cdot W_n .
		\end{array}
	\right. \]
Let $(U_n)_{n\in\integers}$ be a sequence of independent and identically distributed
variables on $\sphere$ with distribution $\mu$.
We assume that $\support(\mu)\times\support(\rho)=\sphere\times(0,+\infty)$ and
	\[ \int\limits_{(0,+\infty)}s\rho(s)\dif s = 1 .\]
The transition function
of the time-homogeneous Markov process $(W_n)_{n\in\integers}$ is given by
	\[ P : \mathscr{B}(\,(0,+\infty)\,)\times(0,+\infty)\to [0,1]:
		P_x(A) = \int\limits_{(0,+\infty)}\chi_A(sx)\rho(s)\dif s .\]
For every continuous and bounded function $f:(0,+\infty)\to\reals$, we
have
	\[ P_xf = \int\limits_{(0,+\infty)}f(xs)\rho(s)\dif s ,\]
which is continuous with respect to the parameter $x\in (0,+\infty)$.
By assumption on the support of $\rho$, every nonempty open set of
$(0,+\infty)$ is reachable with positive probability from every point
$x\in (0,+\infty)$, in one step. We now study the nonevanescence
of $(W_n)_{n\in\integers}$. It remains to prove that the process
$((U_n,W_n))_{n\in\integers}$ is
nonevanescent. Since $\sphere$ is compact, it suffices to shows that $(W_n)_{n\in\integers}$
is nonevanescent in $(0,+\infty)$.
We proceed by following the classical drift criterion~\cite{Tweedie}*{chapters~8,~9}.
By assumption, we have
	\[ \int\limits_{(0,+\infty)}s\,P_x(\dif s) =
		x\int\limits_{(0,+\infty)}s\rho(s)\dif s = x .\]
We write $E = \{ \exists K\subseteq\SymmSpace : K\ \textrm{compact}, (W_n)_{n\in\integers}
			\ \textrm{enters}\ K\ \textrm{infinitely many often} \}\in\mathcal{F}$.
Assume by contradiction
that $\mathbb{P}(E)<1$, and let $K\subseteq (0,+\infty)$ be a compact
set. Then there exists $k\in\integers$ such that
$0 < \mathbb{P}( \{\forall i>k:W_i\notin K\}\setminus E )$.
Denoting by $\mu$ the initial distribution of the process,
the distribution of $W_k$ is given by the measure
	\[ \mu P^k : \mathscr{B}(\SymmSpace) \to [0,1]:
		\mu P^k(A) = \int\limits_\SymmSpace P^k_s(\SymmSpace^{k-1}\times A)\,\mu(\dif s) .\]
Hence, the new process $(G_n)_{n\in\integers}$ defined for every $n\in\integers$ by
$G_n=W_{n+k-1}$ is a time-homogeneous Markov process with transition function $P$,
and initial distribution $\mu P^k$.
The random variable
	\[ T = \min\{ n\in\integers : G_n\in K \} ,\]
is adapted to the filtration $(\mathcal{F}_{n+k-1})_{n\in\integers}$ and
it satisfies $\{T=+\infty\} = \{\forall i\geq k:W_i\notin K\}$.
We use the martingale convergence theorem~\cite{Billingsley}*{Theorem~35.5} to show
that the set $(\{T=+\infty\}\setminus E)\in\mathcal{F}$
has null $\mathbb{P}$-measure~\cite{Tweedie}*{Proposition~9.4.1}.
To see this, observe that the stochastic process
$(M_n)_{n\in\integers}$ defined for every $n\in\integers$ by
	\[ M_n = G_n\chi_{\{T \geq n\}}, \]
is a positive martingale with respect to the filtration $(\mathcal{F}_{n+k-1})_{n\in\integers}$, by construction of $c$.
By the martingale convergence theorem, it converges $\mathbb{P}$-almost-surely
to some $\mathcal{F}$-measurable random variable $M_\infty$, which is $\mathbb{P}$-almost-surely finite. Therefore, we have $\mathbb{P}$-almost-surely
	\[ \chi_{\{T=+\infty\}}M_\infty = \chi_{\{T=+\infty\}} \lim_{n\to+\infty}W_{n+k-1} ,\]
so that $\{T=+\infty\}\setminus E$ has null $\mathbb{P}$-measure, which contradicts the
construction of $k\in\integers$. Therefore, the process $(W_n)_{n\in\integers}$ is
nonevanescent, and so is $((U_n,W_n))_{n\in\integers}$.
We deduce from \cref{corollary.markov.quick}
that for every $X\in\measurable(\ambSp)$, the sequence of successive cap symmetrizations
(resp.~polarizations) $(X^{(U_1,W_1)\dots (U_n,W_n)})_{n\in\integers}$ converges in measure to $X^\star$.
\end{example}

\subsection{Example for Markov cap symmetrizations on $[0,+\infty)$}\label{section.newSymmSpace}
The previous examples concern random walks on $\sphere$ or $\sphere\times(0,+\infty)$, viewed
as symmetrization spaces equipped with a strict asymmetry function. In this section,
we present a \emph{new example} of cap symmetrization space $\CapSymm^\sharp\subset\CapSymm$.
This new symmetrization space satisfies that the subset
	\[ \overline{ \CapSymm^\sharp \setminus \boundary{\CapSymm} }
		\Subset \CapSymm^\sharp \]
is \emph{not} a symmetrization space. In other words, our example occurs in a subset
$\CapSymm^\sharp$ where the nonstrict symmetrization of $\boundary{\CapSymm}$ are
needed for the convergence. To our knowledge,
such symmetrization spaces are unknown from the literature.
We will see that it is possible to construct
a continuous and nonevanescent time-homogeneous
Markov process in $\CapSymm^\sharp$ such that the universal convergence fails.
This will illustrate the necessity of the discrimination condition
\ref{discrimination.proposition.markov} in \cref{proposition.markov}.

As a preparation for the results of this paragraph, we first recall a
known model~\cite{Tweedie} of random walk in $[0,+\infty)$. Although the assumptions in the
following lemma can be weakened, we keep them as simple as possible to make the analysis
easy.
\begin{lemma}[Brownian motion on $[0,+\infty)$]\label{lemma.randomWalkTruncated}
Let $(Z_n)_{n\in\integers}$ be a sequence of
independent and identically distributed variables on $\reals$
with density function $\rho$, and define for every $n\in\integers$
	\[ \left\{
		\begin{array}{ll}
			W_1 &= Z_1, \\
			W_{n+1} &= \max\{W_n + Z_{n+1}, 0\} .
		\end{array}
	\right. \]
If $\support(\rho)=\reals$ and if there exists $\delta>0$ such that
	\[ \int\limits_{(0,+\infty)}s\rho(s)\dif s < 0 < \inf_{[-\delta,\delta]}\rho ,\]
then $(W_n)_{n\in\integers}$ is nonevanescent time-homogeneous Markov process with
strong Feller continuous transition function $P$ such that
for every $s\in [0,+\infty)$, for every nonempty open set $\mathcal{O}\subseteq [0,+\infty)$,
we have
	\[ \sum_{n\in\integers}P^n_s(\,(0,+\infty)^{n-1}\times\mathcal{O}\,) > 0 .\]
\end{lemma}
\begin{proof}
The process $(W_n)_{n\in\integers}$ is a Markov process, whose transition function is
given by
	\[ P : \mathscr{B}(\,[0,+\infty)\,)\times[0,+\infty) \to [0,1] :
		P_s(A) = \Gamma(A\setminus\{0\}-s) + \Gamma(\,(-\infty,-s]\,)\delta_0(A) .\]
Here $\Gamma$ stands for the Lebesgue measure on $\reals$ weighted by $\rho$.
The continuity of translations~\cite{Willem}*{Lemma~4.3.8}
and the dominated convergence theorem
that $P$ is strong Feller continuous. The nonevanscence of $(W_n)_{n\in\integers}$
is known~\cite{Tweedie}*{Proposition~9.4.5,~Theorem~9.4.1} and the proof
is similar to the proof given for the brownian motion on $(0,+\infty)$.
We omit the details.

Let us prove that
for every $s\in [0,+\infty)$, for every nonempty open set $\mathcal{O}\subseteq [0,+\infty)$,
we have
	\[ \sum_{n\in\integers}P^n_s(\,(0,+\infty)^{n-1}\times\mathcal{O}\,) > 0 .\]
First observe that it is sufficient to prove the claim for every nonempty open set of
$(0,+\infty)$. By assumption on $\rho$, there exists $C>0$ such that
for every Borel measurable set $A\subseteq [0,+\infty)$,
	\[ \int\limits_A\rho(x)\dif x \geq C\int\limits_A\chi_{[-\delta,\delta]}\dif x .\]
Therefore, by comparison of series, we can assume without loss of generality
that the distribution $\Gamma$ is the uniform
distribution on the interval $[-\delta,\delta]$. A straightforward computation then shows that
for every $s\in[0,+\infty)$ and for every $n\in\integers$, the measure
	\[ H^n_s : \mathscr{B}(\,(0,+\infty)\,) \to [0,1] : H^n_s(A)=P^n_s(\,(0,+\infty)^{n-1}\times A\,) \]
has support $[\max\{0,s-n\delta\},s+n\delta]$, which concludes the proof.
\end{proof}
\begin{definition}\label{defCapSymmSharp}
Fix $\hat{e}\in\sphere$. The truncated cap symmetrization
space $\CapSymm^\sharp\subseteq\CapSymm$ is defined from
$\{-1,1\}\times [0,+\infty)$ through the action
	\[ \left\{ \begin{array}{l}
		(1,r) \mapsto \textrm{the cap symmetrization}\ (\hat{e},r) \\
		(-1,r) \mapsto \textrm{the cap symmetrization}\ (-\hat{e},0) .
	\end{array} \right. \]
\end{definition}
Observe that $\CapSymm^\sharp$ posses two connected components,
one of them being $\{-1\}\times[0,+\infty)$ whose elements act the same way on
$\measurable(\ambSp)$.
\begin{proposition}
The symmetrization space $\CapSymm^\sharp$ acting
on $(\measurable(\ambSp),d)$ as in \cref{defCapSymmSharp}, is $\star$-compatible.
\end{proposition}
\begin{proof}
The algebraic relation $s\circ\star=\star=\star\circ s$ is valid for every $s\in\CapSymm$, so it is true in $\CapSymm^\sharp$.
Let $A\in\mathscr{B}(\ambSp)$ be a Borel measurable set such that
$\Haus^d(A^s\Delta A)=0$ for every $s\in\CapSymm^\sharp$.
Assume by contradiction that $\Haus^d(A^\star\Delta A)>0$. Consider the Borel measurable set
	\[ B = A^{(-\hat{e},0)(\hat{e},0)} . \]
Since $A=B$ in $\measurable(\ambSp)$, we clearly have
$B^\star=A^\star$, $\Haus^d(B^\star\Delta B)>0$
and $\Haus^d(B^s\Delta B)=0$ for every $s\in\CapSymm^\sharp$.
Moreover, the set $B\cap\partial B(0,t)$ equals $\emptyset$ or $\partial B(0,t)$, for every $t>0$, by construction of cap symmetrizations and by assumption on $A$. Since
$\Haus^d(B)=\Haus^d(B^\star)$, the sets $B\setminus B^\star$ and $B^\star\setminus B$
share the same positive Hausdorff measure. There exists
$x_1=(\hat{e},r_1)\in B^\star\setminus B$
and $x_2=(\hat{e},r_2)\in B\setminus B^\star$, with $r_1,r_2>0$,
and such that for every $r>0$, we have
	\[ \Haus^d(\,B(x_1,r)\cap (B^\star\setminus B)\,) > 0,
		\quad \Haus^d(\,B(x_2,r)\cap (B\setminus B^\star)\,) > 0 .\]
But then we get $\Haus^d(B^{(\hat{e},\frac{r_1+r_2}{2})}\Delta B)>0$, which contradicts the construction
of $B$. Therefore, we should have $\Haus^d(B\Delta B^\star)=0$.
\end{proof}
\begin{example}[Random walk on $\CapSymm^\sharp$]
Consider a random walk $(W_n)_{n\in\integers}$ on $[0,+\infty)$
constructed as in \cref{lemma.randomWalkTruncated}.
We write
	\[ K : [0,+\infty) \times \mathscr{B}(\,[0,+\infty)\,) \to [0,1] \]
the transition function of $(W_n)_{n\in\integers}$.
Consider a Markov process $(U_n)_{n\in\integers}$ on $\{-1,1\}$,
independent of $(W_n)_{n\in\integers}$, and whose transition function
$D$ has the form
	\[ H_i(A)=\beta_i\delta_{\{i\}}(A) + (1-\beta_i)\delta_{\{-i\}}(A) ,\]
where $\beta_1,\beta_{-1}\in (0,1)$. The product process $((U_n,W_n))_{n\in\integers}$
is a Markov process on $\{-1,1\}\times [0,+\infty)$, which satisfies the following
conditions:
	\begin{itemize}
		\item the transition function $P$ of the product process satisfies
			\[ P_{(i,r)}(A) = H_i(\{1\})K_r(\pi^+(A)) + H_i(\{-1\})K_r(\pi^-(A)) ,\]
		where $\pi^-(A),\pi^+(A)\in\mathscr{B}(\,[0,+\infty)\,)$ are the unique
		Borel measurable sets that satisfies
			\[ A\cap \{-1\}\times [0,+\infty) = \{-1\}\times\pi^-(A) ,\]
			and
			\[ A\cap \{1\}\times [0,+\infty) = \{1\}\times\pi^+(A) .\]
		\item $P$ is strong Feller continuous and for every nonempty open set $\mathcal{O}\subseteq\{-1,1\}\times [0,+\infty)$, we have
		for every $(i,r)\in\{-1,1\}\times [0,+\infty)$
			\[ \sum_{n\in\integers}P^n_{(i,r)}(\,
				(\{-1,1\}\times(0,+\infty))^{n-1}\times\mathcal{O} \,) > 0 .\]
		\item the process $((U_n,W_n))_{n\in\integers}$ is nonevanescent, since
		so is $(W_n)_{n\in\integers}$.
	\end{itemize}
According to \cref{corollary.markov.quick}, for every $X\in\measurable(\ambSp)$,
the sequence of successive symmetrizations $(X^{(U_1,W_1)\dots(U_n,W_n)})_{n\in\integers}$
converges almost-surely in measure to $X^\star$.
\end{example}
\begin{counterexample}[Convergence failure for Markov process on $\CapSymm^\sharp$]
We illustrate the necessity of the discrimination condition \ref{discrimination.proposition.markov} in \cref{proposition.markov}.
We define the transition function
	\[ P : \CapSymm^\sharp \times \mathscr{B}(\CapSymm^\sharp)
		\to [0,1] \]
through the formula
	\begin{align*}
		P_{(\hat{e},r)}(A) &= \frac{1}{2} ( \delta_{\hat{e}}\otimes\mu + \delta_{-\hat{e}}\otimes\delta_{0})(A) ,
		&	P_{(-\hat{e},r)}(A) &= \frac{1}{2} ( \delta_{-\hat{e}}\otimes\mu + \delta_{\hat{e}}\otimes\delta_{0} )(A).
	\end{align*}
where $\mu:\mathscr{B}([0,+\infty))\to[0,1]$ is a strictly positive distribution. For
example, one could choose $\mu$ given by a half normal distribution
	\[ \mu(A) = \sqrt{\frac{2}{\pi}} \int\limits_A e^{-\frac{x^2}{2}} \dif x .\]
The continuity and discrimination properties are easily checked. The definition of $P$ forces the process to go on the boundary before to jump
on the other half of $\CapSymm^\sharp$.
Considering half-disks
	\begin{align*}
		D^+ &= \{(z_1,z_2)\in\reals^2: {z_1}^2+{z_2}^2<1,z_2>0 \}, \\
		D^- &=\{(z_1,z_2)\in\reals^2: {z_1}^2+{z_2}^2<1,z_2<0 \} ,
	\end{align*}
it is a straightforward computation to see that the sequences of successive
cap symmetrizations $((D^+)^{S_1\dots S_n})_{n\in\integers}$ and $((D^-)^{S_1\dots S_n})_{n\in\integers}$
alternate between $D^+$ and $D^-$, and the almost-sure convergence does not occur.
\end{counterexample}

\end{document}